\documentclass[12pt, reqno]{amsart}
\usepackage{}
\usepackage{amsfonts}
\usepackage{amsmath, amsthm, amscd, amsfonts, amssymb, graphicx, color, enumerate}
\usepackage[bookmarksnumbered, colorlinks, plainpages]{hyperref}

\newtheorem{theorem}{Theorem}[section]
\newtheorem{lemma}[theorem]{Lemma}
\newtheorem{proposition}[theorem]{Proposition}
\newtheorem{corollary}[theorem]{Corollary}
\theoremstyle{definition}
\newtheorem{definition}[theorem]{Definition}

\theoremstyle{remark}
\newtheorem{remark}[theorem]{Remark}
\numberwithin{equation}{section}

\begin{document}
\setcounter{page}{1}

\title{Asymptotic properties of Banach spaces and coarse quotient maps}

\author{Sheng Zhang}

\address{School of Mathematics, Southwest Jiaotong University, Chengdu, Sichuan 611756, China}
\email{\textcolor[rgb]{0.00,0.00,0.84}{sheng@swjtu.edu.cn}}

\subjclass[2010]{Primary 46B80, 46B06.}

\thanks{The author was supported by the Fundamental Research Funds for the Central Universities, Grant Number 2682017CX060}

\begin{abstract}
We give a quantitative result about asymptotic moduli of Banach spaces under coarse quotient maps. More precisely, we prove that if a Banach space $Y$ is a coarse quotient of a subset of a Banach space $X$, where the coarse quotient map is coarse Lipschitz, then the ($\beta$)-modulus of $X$ is bounded by the modulus of asymptotic uniform smoothness of $Y$ up to some constants. In particular, if the coarse quotient map is a coarse homeomorphism, then the modulus of asymptotic uniform convexity of $X$ is bounded by the modulus of asymptotic uniform smoothness of $Y$ up to some constants.
\end{abstract}

\maketitle

\section{Introduction}

The study of asymptotic geometry of Banach spaces dates back to Milman \cite{Milman1971}, in which he introduced two asymptotic properties that are now known as asymptotic uniform convexity and asymptotic uniform smoothness (cf. \cite{JLPS2002diff}). For a Banach space $X$
and $t>0$. The modulus of asymptotic uniform smoothness of $X$ is defined by
$$\bar{\rho}_X(t):=\sup_{x\in S_X}\inf_{{\rm dim}(X/Y)<\infty}\sup_{y\in S_Y}\|x+ty\|-1,$$
and the modulus of asymptotic uniform convexity of $X$ is defined by
$$\bar{\delta}_X(t):=\inf_{x\in S_X}\sup_{{\rm dim}(X/Y)<\infty}\inf_{y\in S_Y}\|x+ty\|-1.$$
A Banach space $X$ is said to be asymptotically uniformly smooth (AUS for short) if $\lim_{t\to0}\bar{\rho}_X(t)/t\to0$ as $t\to0$,
and it is said to be asymptotically uniformly convex (AUC for short) if $\bar{\delta}_X(t)>0$ for all $0<t\le1$.

In close relation to AUC and AUS is Rolewicz's property ($\beta$) that was originally defined using the terminology of ``drop'' \cite{Rolewicz1987}; later Kutzarova \cite{Kutzarova1991} gave an equivalent definition, according to which one can define a modulus for the property: for a Banach space $X$ and $t\in(0,a]$, where $a\in[1,2]$ is a number that depends only on the space $X$, the ($\beta$)-modulus of $X$ is defined by
$$\bar{\beta}_X(t)=1-\sup\left\{\inf_{n\ge 1}
\left\{\frac{\|x-x_n\|}{2}\right\}:~x,x_n\in B_X,~{\rm sep}(\{x_n\}_{n=1}^\infty)\ge t\right\}.$$
Here ${\rm sep}(\{x_n\}_{n=1}^\infty)$ denotes the separating constant of the sequence $\{x_n\}_{n=1}^\infty$:
$$\text{sep}(\{x_n\}_{n=1}^\infty):=\inf\limits_{i\neq j}\|x_i-x_j\|.$$
A Banach space $X$ is said to have property ($\beta$) if $\bar{\beta}_X(t)>0$ for all $t>0$, and for $p\in(1,\infty)$ we say that the ($\beta$)-modulus of $X$ has power type $p$, or $X$ has property ($\beta_p$), if there is a constant $C>0$ so that $\bar{\beta}_X(t)\ge Ct^p$ for all $t>0$.

A reflexive Banach space that is simultaneously AUC and AUS must has property ($\beta$) \cite{Kutzarova1990}. Conversely, if a Banach space $X$ has property ($\beta$), then it must be reflexive and AUC. More precisely, it was shown in \cite{DKR2016} that $\bar{\beta}_X(t)\le\bar{\delta}_X(2t)$ for all $t\in(0,1/2]$. However, property ($\beta$) does not imply AUS isometrically \cite{Kutzarova1990}, but a Banach space with property ($\beta$) admits an equivalent norm that is AUS. A complete renorming argument of property ($\beta$) can be found in the recent paper by Dilworth, Kutzarova, Lancien and Randrianarivony \cite{DKLRbeta}.

\smallskip

Bates, Johnson, Lindenstrauss, Preiss and Schechtman first studied nonlinear quotient maps in the Banach space setting \cite{BJLPS1999}. A map $f:X\to Y$ between two metric spaces $X$ and $Y$ is said to be co-uniformly continuous if for every $\varepsilon>0$, there exists $\delta=\delta(\varepsilon)>0$ such that for all $x\in X$, $$f(B_X(x,\varepsilon))\supseteq B_Y(f(x),\delta).$$
If $\delta$ can be chosen to be $\varepsilon/C$ for some constant $C>0$ that is independent of $\varepsilon$, then $f$ is said to be co-Lipschitz.
A uniform (resp. Lipschitz) quotient map is a map that is both uniform continuous and co-uniform continuous (resp. Lipschitz and co-Lipschitz), and $Y$ is said to be a uniform (resp. Lipschitz) quotient of $X$ if there exists a uniform (resp. Lipschitz) quotient map from $X$ onto $Y$.

Lima and Randrianarivony \cite{LimaR2012} showed that for $q>p>1$, $\ell_q$ is not a uniform quotient of $\ell_p$. Their proof relies on a technical argument called ``fork argument''. On the other hand, Baudier and Zhang \cite{BZ2016} gave a different proof by estimating the $\ell_p$-distortion of the countably branching trees. The two proofs are based on similar ideas that use the quantification of property $(\beta)$. The theorem below, first appeared in \cite{DKLR2014}, is the quantitative version of the Lima-Randrianarivony result.

\begin{theorem}[\cite{DKLR2014}]\label{uqquan}
Let $X,Y$ be two Banach spaces. $S$ is a subset of $X$ and $f:S\to Y$ is a uniform quotient map that is Lipschitz for large distances. Then there exists constant $C>0$ that depends only on the map $f$ such that for all $0<t\le 1$, $$\bar{\beta}_X(Ct)\le\frac{3}{2}\bar{\rho}_Y(t).$$
\end{theorem}

The main goal of this paper is to give quantitative results of this kind in the coarse category. It should be note that although property $(\beta)$ is preserved under uniform quotient maps up to renorming (cf. \cite{DKR2016} and \cite{DKLRbeta}), one cannot compare $\bar{\beta}_X$ and $\bar{\beta}_Y$ even if $X$ and $Y$ are uniformly homeomorphic. Indeed, \cite{DKR2016} gave an example of two uniformly homeomorphic Banach spaces one of which has property $(\beta_p)$, $p\in(1,\infty)$, while the other does not admit any equivalent norm with property $(\beta_p)$.

Throughout this article all Banach spaces are real and of infinite dimension. For a metric space $X$, $B_X(x,r)$ denotes the closed ball centered at $x$ with radius $r$. If $X$ is a Banach space, we denote by $B_X$ and $S_X$ its closed unit ball and unit sphere, respectively.

\section{Coarse quotient maps}

A map $f:X\to Y$ between two metric spaces $X$ and $Y$ is said to be coarsely continuous if $\omega_f(t)<\infty$ for all $t>0$, where $$\omega_f(t):=\sup\{d_Y(f(x),f(y)): d_X(x,y)\le t\}$$
is the modulus of continuity of $f$. If $X$ is unbounded, one can define for every $s>0$ the Lipschitz constant of $f$ when distances are at least $s$ by
$$\text{Lip}_s(f):=\sup\left\{\frac{d_Y(f(x),f(y))}{d_X(x,y)}:d_X(x,y)\ge s\right\},$$
then for all $t\ge0$ and $s>0$,
$$\omega_f(t)\le\max\{\omega_f(s),\text{Lip}_s(f)\cdot t\}.$$
Let $$\text{Lip}_\infty(f):=\inf_{s>0}\text{Lip}_s(f)=\lim_{s\to\infty}\text{Lip}_s(f).$$
%The map $f$ is said to be Lipschitz for large distances if $\text{Lip}_s(f)<\infty$ for all $s>0$;
The map $f$ is said to be coarse Lipschitz if $\text{Lip}_\infty(f)<\infty$, or equivalently, if $\text{Lip}_s(f)<\infty$ for some $s>0$.

\medskip

The following notion of coarse quotient map was introduced by Zhang \cite{Zhang}:

\begin{definition}[\cite{Zhang}]\label{cqdef}
Let $X,Y$ be two metric spaces. For a constant $K\ge0$, a map $f:X\to Y$ is said to be co-coarsely continuous with constant $K$ if for every $d>K$, there exists $\delta=\delta(d)>0$ such that for all $x\in X$,
$$f(B_X(x,\delta))^K\supseteq B_Y(f(x),d).$$
Here for a subset $A$ of $X$, $A^K$ denotes the $K$-neighborhood of $A$, that is, $$A^K:=\{x\in X: d_X(x,a)\le K~\text{for some}~a\in A\}.$$
If $f$ is both coarsely continuous and co-coarsely continuous (with constant $K$), then we say $f$ is a coarse quotient map (with constant $K$). $Y$ is said to be a coarse quotient of $X$ if there exists a coarse quotient map from $X$ to $Y$.
\end{definition}

Recall that a metric space $X$ is said to be metrically convex if for every $x_0, x_1\in X$ and $0<\lambda<1$, there is a point $x_\lambda\in X$ such that
$$d(x_0,x_\lambda)=\lambda d(x_0, x_1)\hspace{4mm}\text{and}\hspace{4mm}d(x_1, x_\lambda)=(1-\lambda)d(x_0, x_1).$$
It is well-known that a coarsely continuous map defined on a metrically convex space must be Lipschitz for large distances. Similarly, if the range space of a co-coarsely continuous map with constant $K$ is metrically convex, then the map is ``co-Lipschitz for large distantces with constant $K$'' as stated in the Lemma below.

\begin{lemma}\label{cqLiplarge}
Let $X,Y$ be two metric spaces and assume that $Y$ is metrically convex. If $f:X\to Y$ is a co-coarsely continuous map with constant $K$, then for every $d>K$, there exists $c=c(d,K)>0$ such that for all $x\in X$ and $r\ge d$,
\begin{align}
f(B_X(x,cr))^K\supseteq B_Y(f(x),r).\label{cqLiplargeineq}
\end{align}
\end{lemma}

\begin{proof}
For $x\in X$ and $r\ge d$, let $n:=\lfloor\frac{r}{d-K}\rfloor+1$. Then for every $y\in B_Y(f(x),r)$, $d_Y(y,f(x))\le r<n(d-K)$. By the metric convexity of $Y$, one can find points $\{u_i\}_{i=0}^n$ in $Y$ with $u_0=f(x)$ and $u_n=y$ such that $d_Y(u_i,u_{i-1})<d-K$, $i=1,...,n$. Since $f$ is co-coarsely continuous with constant $K$, we have $$u_1\in B_Y(f(x),d)\subseteq f(B_X(x,\delta))^K,$$
where $\delta=\delta(d)>0$ is given by Definition \ref{cqdef}, so there exists $z_1\in B_X(x,\delta)$ so that $d_Y(u_1,f(z_1))\le K$. This implies, by the triangle inequality, that $u_2\in B_Y(f(z_1),d)$. Again the co-coarse continuity of $f$ guarantees that there is $z_2\in B_X(z_1,\delta)$ that satisfies $d_Y(u_2,f(z_2))\le K$. Repeat the procedure $n$ times we get points $\{z_i\}_{i=0}^n$ in $X$, where $z_0=x$, with the following property: $d_X(z_i,z_{i-1})\le\delta$ and $d_Y(u_i,f(z_i))\le K$, $i=1,...,n$. It follows that $z_n\in B_X(x,n\delta)$, hence $y\in f(B_X(x,n\delta))^K$. Note that $n\le\left(\frac{1}{d-K}+\frac{1}{d}\right)r$, thus \eqref{cqLiplargeineq} follows by putting $c=\left(\frac{1}{d-K}+\frac{1}{d}\right)\delta$.
\end{proof}

\begin{remark} Lemma \ref{cqLiplarge} is an improvement of Lemma 3.2 in \cite{Zhang}, where $d>2K$ is required. Also, For $d>K$, it follows from \eqref{cqLiplargeineq} that the constant $c=c(d,K)>0$ satisfies for all $x\in X$ and $r>0$,
$$f(B_X(x,cr))^d\supseteq B_Y(f(x),r).$$
It means that co-coarsely continuous maps are co-Lipschitz with a slightly larger constant if the range space is metrically convex.
\end{remark}

Under the assumption of Lemma \ref{cqLiplarge}, for $d>K$, let $c_d$ be the infimum of all $c$ that satisfy \eqref{cqLiplargeineq} for all $x\in X$ and $r\ge d$. Then $\{c_d\}_{d>K}$ is non-increasing and bounded below by 0, hence converges. Denote $c_\infty(f):=\inf_{d>K}c_d=\lim_{d\to\infty}c_d$.

\begin{lemma}\label{clowbound}
Let $X,Y$ be two metric spaces and assume that $Y$ is metrically convex and unbounded. If $f:X\to Y$ is a coarse quotient map that is coarse Lipschitz, then $${\rm Lip}_\infty(f)c_\infty(f)\ge1.$$
\end{lemma}

\begin{proof}
Let $f$ be a coarse quotient map with constant $K$. First observe $X$ is unbounded and $\text{Lip}_s(f)>0$ for all $s>0$, since otherwise $Y=f(X)^K$ is bounded.

Now by Lemma \ref{cqLiplarge}, for $d>K$, there exists $c=c(d,K)>0$ such that for all $x\in X$ and $r\ge d$,
$$B_Y(f(x),r)\subseteq f(B_X(x,cr))^K\subseteq B_Y(f(x),\omega_f(cr))^K=B_Y(f(x),\omega_f(cr)+K),$$
and this implies that $r\le\omega_f(cr)+K$. Since $f$ is coarse Lipschitz, let $s>0$ be such that $0<\text{Lip}_s(f)<\infty$. Then for $t\ge s$ one has
$$r\le\omega_f(cr)+K\le\max\{\omega_f(t),\text{Lip}_t(f)\cdot cr\}+K.$$ Choose large $r$ so that $\text{Lip}_t(f)\cdot cr>\omega_f(t)$, then $r\le\text{Lip}_t(f)\cdot cr+K$, so $$\text{Lip}_t(f)\cdot c\ge\frac{r-K}{r}.$$
Let $r\to\infty$, it follows that $\text{Lip}_t(f)\cdot c\ge1$, and then we finish the proof by letting $t\to\infty$.

\end{proof}

\section{Quantitative results under coarse quotient maps}

Before stating our main theorem, we need the following alternative definition for the modulus of AUS that may be known to experts, but we still give a proof here since we could not find one in the literature.

\begin{proposition}\label{AUSdef}
Let $X$ be a Banach space and $0<t\le1$. Then
\begin{align}\label{AUSmod}
\bar{\rho}_X(t)=\sup_{x\in B_X}\inf_{{\rm dim}(X/Y)<\infty}\sup_{y\in S_Y}\|x+ty\|-1
\end{align}
\end{proposition}

\begin{proof}
First we show that for every $x,y\in X$ the function $$f(\lambda)=\max\{\|\lambda x+y\|,\|\lambda x-y\|\}$$ is nondecreasing on $(0,\infty)$. Let $0<\lambda_1<\lambda_2$, we may assume that $\|\lambda_1 x+y\|\ge\|\lambda_1 x-y\|$ and let $x^*\in S_{X^*}$ be such that $x^*(\lambda_1 x+y)=\|\lambda_1 x+y\|$. Then $x^*(x)\ge0$, since otherwise
$$\|\lambda_1 x+y\|\ge\|\lambda_1 x-y\|\ge(-x^*)(\lambda_1 x-y)>x^*(\lambda_1 x+y)=\|\lambda_1 x+y\|.$$
Therefore, $f(\lambda_1)=x^*(x)\lambda_1+x^*(y)\le x^*(x)\lambda_2+x^*(y)\le\|\lambda_2 x+y\|\le f(\lambda_2)$.

\smallskip

Now we prove \eqref{AUSmod}. Let $0<t\le1$. If $x=0$ then
$$\inf_{{\rm dim}(X/Y)<\infty}\sup_{y\in S_Y}\|x+ty\|=t\le\bar{\rho}_X(t)+1.$$
For $x\in B_X\backslash\{0\}$ one has
\begin{align*}
\inf_{{\rm dim}(X/Y)<\infty}\sup_{y\in S_Y}\|x+&ty\|=\inf_{{\rm dim}(X/Y)<\infty}\sup_{y\in S_Y}\max\{\|x+ty\|,\|x-ty\|\}\\
&\le\inf_{{\rm dim}(X/Y)<\infty}\sup_{y\in S_Y}\max\left\{\left\|\frac{x}{\|x\|}+ty\right\|,\left\|\frac{x}{\|x\|}-ty\right\|\right\}\\
&=\inf_{{\rm dim}(X/Y)<\infty}\sup_{y\in S_Y}\left\|\frac{x}{\|x\|}+ty\right\|\\
&\le\bar{\rho}_X(t)+1,
\end{align*}
thus \eqref{AUSmod} follows.
\end{proof}

\begin{theorem}\label{mainthm1}
Let $X,Y$ be two Banach spaces. $S$ is a subset of $X$ and $f:S\to Y$ is a coarse quotient map that is coarse Lipschitz. Then for all $0<t\le 1$, $$\bar{\beta}_X\left(\frac{t}{48{\rm Lip}_\infty(f)c_\infty(f)}\right)\le\frac{3}{2}\bar{\rho}_Y(t).$$
\end{theorem}

\begin{proof}
Since $f:S\to Y$ is a coarse quotient map that is coarse Lipschitz, it follows from Lemma \ref{clowbound} that $0<\text{Lip}_\infty(f)<\infty$. Choose $s>0$ such that $\text{Lip}_s(f)<2\text{Lip}_\infty(f)$. For $0<t\le1$, one has $0\le\bar{\rho}_Y(t)\le t\le1$. Let $\varepsilon>0$ be small so that
$$\varepsilon<\min\left\{\frac{1}{2}, \frac{2-\bar{\rho}_Y(t)}{6\text{Lip}_\infty(f)+2}, c_\infty(f)\right\},$$
and choose large $d$ that satisfies
$$d>\max\left\{\frac{3K}{\varepsilon},\frac{12(2K+\omega_f(s))}{t}\right\}\hspace{4mm}\text{and}\hspace{4mm}c_{d/3}<c_\infty(f)+\varepsilon.$$
Since $c_\infty(f)-\varepsilon<c_d$, there exist $z_\varepsilon\in S$ and $R\ge d$ such that
$$B_Y(f(z_\varepsilon),R)\nsubseteq f(B_S(z_\varepsilon,R(c_\infty(f)-\varepsilon))^K,$$
so there is $y_\varepsilon\in Y$ with $0<\|y_\varepsilon-f(z_\varepsilon)\|\le R$ such that
\begin{align}\label{emptyintersection}
B_Y(y_\varepsilon,K)\cap f(B_S(z_\varepsilon,R(c_\infty(f)-\varepsilon))=\emptyset.
\end{align}

Now cut the line segment $[y_\varepsilon,f(z_\varepsilon)]$ into three equal pieces, namely, let $m,M\in Y$ be such that $m-f(z_\varepsilon)=M-m=y_\varepsilon-M$, then
$$m\in B_Y\left(f(z_\varepsilon),\frac{R}{3}\right)\subseteq f\left(B_S\left(z_\varepsilon,\frac{R}{3}(c_\infty(f)+\varepsilon)\right)\right)^K,$$
so there is $x\in S$ such that
$$\|x-z_\varepsilon\|\le\frac{R}{3}(c_\infty(f)+\varepsilon)\hspace{4mm}\text{and}\hspace{4mm}\|m-f(x)\|\le K.$$

By the definition of $\bar{\rho}_Y(t)$ (Proposition \ref{AUSdef}), there exists a finite-codimentional subspace $Z$ of $Y$ so that
$$\sup_{z\in S_Z}\left\|M-m+\frac{tR}{3}z\right\|<\frac{R}{3}(1+\bar{\rho}_Y(t)+\varepsilon).$$
Set $y_n:=M+\frac{tR}{3}e_n$, where $(e_n)$ is a basic sequence in $S_Z$ with basis constant less than $2$. Then
$$\|y_n-m\|=\left\|M-m+\frac{tR}{3}e_n\right\|<\frac{R}{3}(1+\bar{\rho}_Y(t)+\varepsilon),$$
and by the triangle inequality,
$$\|y_n-f(x)\|<\frac{R}{3}(1+\bar{\rho}_Y(t)+\varepsilon)+K<\frac{R}{3}(1+\bar{\rho}_Y(t)+2\varepsilon).$$
Thus
$$y_n\in B_Y\left(f(x),\frac{R(1+\bar{\rho}_Y(t)+2\varepsilon)}{3}\right)\subseteq
f\left(B_S\left(x,\frac{R(1+\bar{\rho}_Y(t)+2\varepsilon)(c_\infty(f)+\varepsilon)}{3}\right)\right)^K,$$
so there exists $z_n\in S$ such that
$$\|z_n-x\|\le\frac{R}{3}(1+\bar{\rho}_Y(t)+2\varepsilon)(c_\infty(f)+\varepsilon)\hspace{4mm}\text{and}\hspace{4mm}\|y_n-f(z_n)\|\le K.$$
Note that
$$\|y_\varepsilon-y_n\|=\left\|y_\varepsilon-M-\frac{tR}{3}e_n\right\|
=\left\|M-m-\frac{tR}{3}e_n\right\|<\frac{R}{3}(1+\bar{\rho}_Y(t)+\varepsilon),$$
so again by the triangle inequality,
$$\|y_\varepsilon-f(z_n)\|<\frac{R}{3}(1+\bar{\rho}_Y(t)+\varepsilon)+K<\frac{R}{3}(1+\bar{\rho}_Y(t)+2\varepsilon),$$
hence
$$y_\varepsilon\in B_Y\left(f(z_n),\frac{R(1+\bar{\rho}_Y(t)+2\varepsilon)}{3}\right)\subseteq
f\left(B_S\left(z_n,\frac{R(1+\bar{\rho}_Y(t)+2\varepsilon)(c_\infty(f)+\varepsilon)}{3}\right)\right)^K,$$
and this gives $x_n\in S$ that satisfies
$$\|x_n-z_n\|\le\frac{R}{3}(1+\bar{\rho}_Y(t)+2\varepsilon)(c_\infty(f)+\varepsilon)\hspace{4mm}\text{and}\hspace{4mm}
\|y_\varepsilon-f(x_n)\|\le K.$$
On the other hand, in view of \eqref{emptyintersection}, one has $\|z_\varepsilon-x_n\|>R(c_\infty(f)-\varepsilon)$, so
\begin{align*}
\|z_\varepsilon-z_n\|
&\ge\|z_\varepsilon-x_n\|-\|x_n-z_n\|\\
&>R(c_\infty(f)-\varepsilon)-\frac{R}{3}(1+\bar{\rho}_Y(t)+2\varepsilon)(c_\infty(f)+\varepsilon)\\
&=\frac{R}{3}(c_\infty(f)+\varepsilon)\left(\frac{3(c_\infty(f)-\varepsilon)}{c_\infty(f)+\varepsilon}-1-\bar{\rho}_Y(t)-2\varepsilon\right)\\
&\ge\frac{R}{3}(c_\infty(f)+\varepsilon)\left(3\left(1-\frac{2\varepsilon}{c_\infty(f)}\right)-1-\bar{\rho}_Y(t)-2\varepsilon\right)\\
&\ge\frac{R}{3}(c_\infty(f)+\varepsilon)(2-\bar{\rho}_Y(t)-(6\text{Lip}_\infty(f)+2)\varepsilon)>0.
\end{align*}

For $n,k\in\mathbb{N}$ with $n\neq k$,
$$\|y_n-y_k\|=\frac{tR}{3}\|e_n-e_k\|>\frac{tR}{6}>\omega_f(s)+2K,$$
and also note that
\begin{align*}
\|y_n-y_k\|&\le\|y_n-f(z_n)\|+\|f(z_n)-f(z_k)\|+\|y_k-f(z_k)\|\\
&\le2K+\omega_f(\|z_n-z_k\|),
\end{align*}
so $\omega_f(\|z_n-z_k\|)>\omega_f(s)$, thus $\|z_n-z_k\|>s$. It follows that
\begin{align*}
\frac{tR}{6}<\|y_n-y_k\|&\le\|y_n-f(z_n)\|+\|f(z_n)-f(z_k)\|+\|y_k-f(z_k)\|\\
&\le2K+\text{Lip}_s(f)\|z_n-z_k\|\\
&<\frac{tR}{12}+2\text{Lip}_\infty(f)\|z_n-z_k\|,
\end{align*}
hence $\|z_n-z_k\|>tR/24\text{Lip}_\infty(f)$.

\bigskip

In summary, for $n,k\in\mathbb{N}$ with $n\neq k$ we have the following:
$$\|z_n-z_k\|>\frac{tR}{24\text{Lip}_\infty(f)},~
\|z_\varepsilon-z_n\|>\frac{R(c_\infty(f)+\varepsilon)(2-\bar{\rho}_Y(t)-(6\text{Lip}_\infty(f)+2)\varepsilon)}{3},$$
$$\|z_{\varepsilon}-x\|\le\frac{R}{3}(c_\infty(f)+\varepsilon),\hspace{5mm}
\|z_n-x\|\le\frac{R}{3}(c_\infty(f)+\varepsilon)(1+\bar{\rho}_Y(t)+2\varepsilon).$$
Since
$$\frac{t}{48\text{Lip}_\infty(f)c_\infty(f)}
\le\frac{tR}{24\text{Lip}_\infty(f)}\cdot\frac{1}{\frac{R}{3}(c_\infty(f)+\varepsilon)(1+\bar{\rho}_Y(t)+2\varepsilon)}
\le\frac{t}{8\text{Lip}_\infty(f)c_\infty(f)}$$
and
$$1-\frac{1}{2}\cdot\frac{\frac{R}{3}(c_\infty(f)+\varepsilon)(2-\bar{\rho}_Y(t)-(6\text{Lip}_\infty(f)+2)\varepsilon)}
{\frac{R}{3}(c_\infty(f)+\varepsilon)(1+\bar{\rho}_Y(t)+2\varepsilon)}
\le\frac{3}{2}\bar{\rho}_Y(t)+(3\text{Lip}_\infty(f)+3)\varepsilon,$$
it follows from the definition of $(\beta)$-modulus that
$$\bar{\beta}_X\left(\frac{t}{48\text{Lip}_\infty(f)c_\infty(f)}\right)\le\frac{3}{2}\bar{\rho}_Y(t)+(3\text{Lip}_\infty(f)+3)\varepsilon.$$
The proof is complete by letting $\varepsilon\to0$.

\end{proof}

It is easy to compute that for $1<p<\infty$ and $0<t\le1$,
\begin{align*}
&\bar{\delta}_{\ell_p}(t)=\bar{\rho}_{\ell_p}(t)=(1+t^p)^{\frac{1}{p}}-1,\\
&\bar{\delta}_{c_0}(t)=\bar{\rho}_{c_0}(t)=0,
\end{align*}
and since $\ell_p$ has property ($\beta_p$) (see \cite{betaMod} for the explicit formula of ($\beta$)-modulus of $\ell_p$), we can recover the main result of \cite{Zhang} as an immediate consequence of Theorem \ref{mainthm1}.

\begin{corollary}\mbox{}\par
\begin{enumerate}[(i)]
\item $\ell_q$ is not a coarse quotient of $\ell_p$ for $1<p<q<\infty$.
\item $c_0$ is not a coarse quotient of any Banach space with property ($\beta$).
\end{enumerate}
\end{corollary}

\section{Quantitative results under coarse homeomorphisms}

This section is devoted to a special case of Theorem \ref{mainthm1} when the coarse quotient map $f$ is a coarse homeomorphism. Recall that a coarsely continuous map $f:X\to Y$ between two metric spaces $X$ and $Y$ is called a coarse homeomorphism if there exists another coarsely continuous map $g:Y\to X$ such that
$$\sup_{x\in X}d_X(g\circ f(x),x)<\infty\hspace{5mm}\text{and}\hspace{5mm}\sup_{y\in Y}d_Y(f\circ g(y),y)<\infty.$$
It was proved in \cite{Zhang} that a coarse homeomorphism is necessarily a coarse quotient map.

The main tool we need is approximate metric midpoint, which was first used by Enflo (unpublished) to show that $L_1$ is not uniformly homeomorphic to $\ell_1$ (see, e.g., \cite{Benyasurvey}). Given two points $x,y$ in a metric space $X$ and $\delta\in (0,1)$, the set of $\delta$-approximate metric midpoints between $x$ and $y$ is defined by
$${\rm Mid}(x,y,\delta):=\left\{z\in X: \max\{d_X(z,x),d_X(z,y)\}\le\frac{1+\delta}{2}d_X(x,y)\right\}.$$
The lemma below is sometimes known as the ``stretching lemma''.

\begin{lemma}[\cite{KaltonLova2008}]
Let $f:X\to Y$ be a coarse Lipschitz map from an unbounded metric space $X$ to a metric space $Y$. If ${\rm Lip}_\infty(f)>0$ then for any $d>0$, any $\varepsilon>0$ and any $0<\delta<1$ there exist $x,y\in X$ with $d_X(x,y)\ge d$ such that
$$f({\rm Mid}(x,y,\delta))\subseteq {\rm Mid}(f(x),f(y),(1+\varepsilon)\delta).$$
\end{lemma}

The next lemma, which can be found in \cite{LovaPhD}, relates the set of approximate metric midpoints in a Banach space with the moduli of AUC and AUS of the space.

\begin{lemma}[\cite{LovaPhD}]\label{midmoduli}
Let $X$ be a Banach space, $x\in S_X$ and $0<t\le1$.
\begin{enumerate}[(i)]
\item For every $\varepsilon>0$, there exists a finite-codimentional subspace $Y$ of $X$ such that
$$tB_Y\subseteq {\rm Mid}(x,-x,\bar{\rho}_X(t)+\varepsilon).$$
\item If $\bar{\delta}_X(t)>0$, then for every $0<\varepsilon<1$, there exists a compact subset $K$ of $X$ such that $${\rm Mid}(x,-x,(1-\varepsilon)\bar{\delta}_X(t))\subseteq K+3tB_X.$$
\end{enumerate}
\end{lemma}

We also need the following easy lemma.

\begin{lemma}\label{Goreliklike}
Let $f: X\to Y$ be a map between Banach spaces $X$ and $Y$. If there exist a closed ball $B_r$ of radius $r$ in $X$, a closed ball $B_s$ of radius $s$ in $Y$ and a compact set $K\subseteq Y$ such that $$f(B_r)\subseteq K+B_s,$$ then the compression modulus $\varphi_f$ of $f$ satisfies $$\varphi_f(r):=\inf\{\|f(x)-f(y)\|: \|x-y\|\ge r\}\le 2s.$$
\end{lemma}

\begin{proof}
Choose a $r$-separating sequence $\{x_n\}_{n=1}^\infty$ in $B_r$ and let $f(x_n)=z_n+y_n$, where $z_n\in K$ and $y_n\in B_s$ for all $n\in\mathbb{N}$. For $\varepsilon>0$, since $K$ is compact, by passing to a subsequence we may assume that $\|z_n-z_m\|<\varepsilon$ for all $m,n\in\mathbb{N}$. Then for $m\neq n$,
$$2s\ge\|y_n-y_m\|\ge\|f(x_n)-f(x_m)\|-\|z_n-z_m\|\ge\varphi_f(r)-\varepsilon,$$
and we are done by letting $\varepsilon\to 0$.
\end{proof}

Theorem \ref{cLipLova} below is due to Randrianarivony \cite{LovaPhD}. We present here a proof with improved constants.

\begin{theorem}\label{cLipLova}
Let $X$ and $Y$ be two Banach spaces and $f:X\to Y$ a coarse Lipschitz embedding, i.e., there exist $d\ge0$ and $L,C>0$ such that for all $x,y\in X$ with $\|x-y\|\ge d$, $$\frac{1}{C}\|x-y\|\le\|f(x)-f(y)\|\le L\|x-y\|.$$
Then for all $0<t\le1$, $$\bar{\delta}_Y\left(\frac{t}{7LC}\right)\le\bar{\rho}_X(t).$$
\end{theorem}

\begin{proof}
Let $0<t\le1$ be such that $\bar{\delta}_Y(t/7LC)>0$ and $0<\varepsilon<1/2$. Note that $1/C\le{\rm Lip}_\infty(f)\le L$, one can apply the stretching lemma to find $u,v\in X$ with $\|u-v\|\ge2d/t$ such that
$$f\left({\rm Mid}\left(u,v,(1-2\varepsilon)\bar{\delta}_Y\left(\frac{t}{7LC}\right)\right)\right)\subseteq
{\rm Mid}\left(f(u),f(v),(1-\varepsilon)\bar{\delta}_Y\left(\frac{t}{7LC}\right)\right).$$
By Lemma \ref{midmoduli} (ii), there exists a compact set $K\subseteq Y$ such that
\begin{align*}
{\rm Mid}\left(f(u),f(v),(1-\varepsilon)\bar{\delta}_Y\left(\frac{t}{7LC}\right)\right)&\subseteq K+\frac{3t}{14LC}\|f(u)-f(v)\|B_Y\\
&\subseteq K+\frac{3t}{14C}\|u-v\|B_Y.
\end{align*}
Assume that there exists $\tau>0$ that saisfies $$(1-2\varepsilon)\bar{\delta}_Y\left(\frac{t}{7LC}\right)>\bar{\rho}_X(t)+\tau,$$
then by Lemma \ref{midmoduli} (i), there exists a finite-codimensional subspace $Z$ of $X$ such that
$$f\left({\rm Mid}\left(u,v,(1-2\varepsilon)\bar{\delta}_Y\left(\frac{t}{7LC}\right)\right)\right)\supseteq
f\left(\frac{u+v}{2}+\frac{t\|u-v\|}{2}B_Z\right),$$
thus $$f\left(\frac{u+v}{2}+\frac{t\|u-v\|}{2}B_Z\right)\subseteq K+\frac{3t}{14C}\|u-v\|B_Y.$$
Now it follows from Lemma \ref{Goreliklike} that
$$\frac{t\|u-v\|}{2C}\le\varphi_f\left(\frac{t\|u-v\|}{2}\right)\le\frac{3t}{7C}\|u-v\|,$$
a contradiction. Therefore, we must have
$$(1-2\varepsilon)\bar{\delta}_Y\left(\frac{t}{7LC}\right)\le\bar{\rho}_X(t).$$
We then finish the proof by letting $\varepsilon\to0$.
\end{proof}

\begin{theorem}\label{mainthm2}
Let $X,Y$ be two Banach spaces. $S$ is a subset of $X$ and $f:S\to Y$ is a coarse homeomorphism that is coarse Lipschitz. Then for all $0<t\le 1$, $$\bar{\delta}_X\left(\frac{t}{56{\rm Lip}_\infty(f)c_\infty(f)}\right)\le\bar{\rho}_Y(t).$$
\end{theorem}

\begin{proof}
Let $g:Y\to S$ be a coarsely continuous map such that
$$\sup_{x\in S}d_X(g\circ f(x),x):=M<\infty\hspace{4mm}\text{and}\hspace{4mm}\sup_{y\in Y}d_Y(f\circ g(y),y):=K<\infty.$$
We claim that $g$ is a coarse Lipschitz embedding from $Y$ into $X$.
Indeed, it follows from Proposition 2.5 in \cite{Zhang} that $f$ is a coarse quotient map with constant $K$. Choose $s>2K$ such that
${\rm Lip}_s(f)<2{\rm Lip}_\infty(f)$, and let $d>6K$ be such that $\varphi_g(d)>s$. For $y_1,y_2\in Y$ with $\|y_1-y_2\|\ge d$, one has $\|g(y_1)-g(y_2)\|\ge\varphi_g(d)>s$ and thus
$$\|f\circ g(y_1)-f\circ g(y_2)\|\le2{\rm Lip}_\infty(f)\|g(y_1)-g(y_2)\|.$$
By the triangle inequality,
$$\|f\circ g(y_1)-f\circ g(y_2)\|\ge\|y_1-y_2\|-2K\ge\frac{2}{3}\|y_1-y_2\|,$$
so we obtain that $$\frac{1}{3{\rm Lip}_\infty(f)}\|y_1-y_2\|\le\|g(y_1)-g(y_2)\|.$$
On the other hand, we could make $d$ even larger so that $c_{d-K}<2c_\infty(f)$ and $$\frac{d}{3}\cdot c_\infty(f)>\omega_g(K)+M.$$
Note that
$$\|y_1-y_2\|+K\ge r:=\|y_1-f\circ g(y_2)\|\ge\|y_1-y_2\|-K\ge d-K>K,$$
it follows from Lemma \ref{cqLiplarge} and the definition of $c_{d-K}$ that
$$y_1\in B_Y(f\circ g(y_2),r)\subseteq f(B_S(g(y_2),2rc_\infty(f)))^K,$$
so there exists $x\in S$ such that
$$\|x-g(y_2)\|\le2rc_\infty(f)\hspace{5mm}\text{and}\hspace{5mm}\|y_1-f(x)\|\le K.$$
Now again by the triangle inequality,
\begin{align*}
\|g(y_1)-g(y_2)\|&\le\|g(y_1)-g\circ f(x)\|+\|g\circ f(x)-x\|+\|x-g(y_2)\|\\
&\le\omega_g(K)+M+2rc_\infty(f)\\
&\le\frac{c_\infty(f)}{3}\|y_1-y_2\|+2c_\infty(f)(\|y_1-y_2\|+K)\\
&\le\frac{8c_\infty(f)}{3}\|y_1-y_2\|.
\end{align*}
Therefore, for sufficiently large $d$, one has
$$\frac{1}{3{\rm Lip}_\infty(f)}\|y_1-y_2\|\le\|g(y_1)-g(y_2)\|\le\frac{8c_\infty(f)}{3}\|y_1-y_2\|$$
whenever $\|y_1-y_2\|\ge d$. The desired inequality then follows from Theorem \ref{cLipLova}.
\end{proof}

\begin{remark}
In connection with the modulus of asymptotic uniform convexity, the asymptotic midpoint uniform convexity modulus of a Banach space $X$ was introduced in \cite{DKRRZ2016} as follows:
$$\tilde{\delta}_X(t):=\inf_{x\in S_X}\sup_{{\rm dim}(X/Y)<\infty}\inf_{y\in S_Y}\max\{\|x+ty\|,\|x-ty\|\}-1.$$
A Banach space $X$ is said to be asymptotically midpoint uniformly convex (AMUC for short) if $\tilde{\delta}_X(t)>0$ for all $0<t\le1$. It was implicitly proved in \cite{DKRRZ2016} that Lemma \ref{midmoduli} (ii) still holds true for the AMUC modulus. Therefore, Theorem \ref{cLipLova} and Theorem \ref{mainthm2} can be strengthened by replacing $\bar{\delta}_X$ with $\tilde{\delta}_X$.
\end{remark}

\begin{corollary}\mbox{}\par
\begin{enumerate}[(i)]
\item $\ell_q$ does not coarse Lipschitz embed into $\ell_p$ for $1<p<q<\infty$.
\item $c_0$ does not coarse Lipschitz embed into any AMUC Banach space.
\end{enumerate}
\end{corollary}

\noindent
{\bf Acknowledgement.} The author thanks William B. Johnson for the proof of Proposition \ref{AUSdef}.

\bibliographystyle{amsplain}

\end{document}